\documentclass[11pt]{article}

\usepackage{amsmath,amssymb,amsthm}
\usepackage{graphicx}
\usepackage{geometry}
\usepackage{setspace}
\usepackage{placeins}
\numberwithin{equation}{section}
\usepackage{cite}
\usepackage{tabularx}
\usepackage{float}
\usepackage{comment}
\usepackage{booktabs} 
\newtheorem{theorem}{Theorem}[section]
\begin{document}

\title{\textbf{Analysis of a Stochastic Energy Supply and Demand Model with Renewable Integration}}

\author{S.~O.~Edeki$^1$, S.~Noeiaghdam$^{2,3}$, L. Hairong$^4$, S. Kaennakham$^5$, I. D. Ezekiel$^{6}$ }
  \date{}
 \maketitle
 
\begin{center}
\scriptsize{ $^1$Department of Mathematics, Dennis Osadebay University, Asaba, 320242, Nigeria. soedeki247@gmail.com\\
 $^2$Institute of Mathematics, Henan Academy of Sciences, Zhengzhou 450046, China. snoei@hnas.ac.cn\\
 $^3$Department of Mathematical Sciences, Saveetha School of Engineering, SIMATS, Chennai, 602105, India.\\
 $^4$School of Science, China University of Geosciences, Beijing, PR China. lianhr@cugb.edu.cn \\
 $^5$School of Mathematics and Geoinformatics, Institute of Science, Suranaree University of Technology
Nakhon Ratchasima, THAILAND. sayan$_{-}$kk@g.sut.ac.th  \\
$^6$ Department of Mathematics and Statistics, Federal Polytechnic, Ilaro, Nigeria. imekela.ezekiel@federalpolyilaro.edu.ng  \\

}
\end{center}

\begin{abstract}
In this work, a stochastic energy supply–demand model with renewable integration is developed and analyzed. The basic nonlinear deterministic model describing the relationship among regional demand, external supply, energy imports, and renewable resource integration is extended to an Itô-type stochastic system that captures the uncertainties due to market volatility, climatic variation, policy interventions, and technical changes. Also, the noise structure is multiplicative, ensuring proportional fluctuations and preservation of nonnegativity of the state variables. Global existence, uniqueness of positive solutions, moment boundedness, and stochastic persistence are established rigorously. Furthermore, the deterministic system is analyzed, and stochastic stability is examined using matrix inequality criteria to guarantee almost sure exponential stability of the system in the stochastic setting. Among other results, stochastic perturbations significantly alter the effective system capacity compared to the deterministic case; however, under suitable parameter conditions, boundedness and stability cases are preserved. The Euler–Maruyama scheme is employed to perform numerical simulations to illustrate various dynamical behaviors and highlight the effects of uncertainty on system dynamics.The numerical reliability of the proposed model is further confirmed by additional numerical experiments via the Milstein scheme and parameter sensitivity analysis. Moreover, the results indicate that stochastic effects should be considered for capturing complex energy systems' behavior under uncertainty and its implications for renewable integration.
\end{abstract}

\\


\section{Introduction}

Energy systems are inherently complex and are influenced by interactions among economic, environmental, and technical factors from economic to environmental and technical issues. Changes in market prices, policy interventions, supply chain disruptions, and climatic variations create traces of uncertainty on energy demands, external supplies, etc. [1-3]. Globally, ongoing energy transitions and increased penetration of renewables have brought into focus an important consequence that can be studied more under environmental policy dynamics, as uncertainty is inherently irreducible [4, 5].

Classical models of energy that assume determinism constitute well-structured models that describe equilibrium behavior, growth mechanisms, and conditionally stable points. These models usually take the form of a system of nonlinear differential equations to depict interactions between changes in demand, responses on the supply side, and adoption of renewables [6-8]. Deterministic models, however, are merely mathematically tractable solutions for qualitative analysis and assume perfect predictability and time-independent parameters. Nevertheless, in reality, energy markets are subject to stochastic shocks and continue to display periods of persistent volatility that cannot be fully captured by a deterministic dynamic structure alone [9, 10].

In mathematical terms, stochastic differential equations provide a natural environment for including uncertainty in dynamic systems. In economics and finance, stochastic modeling has been employed to analyze volatility, long term persistence, and stability under random perturbations. For example, stochastic modeling can be beneficial in studying noise-induced transitions, stabilization effects, and eventually changes in equilibrium behavior in power systems [11-13]. Several studies have looked into introducing stochasticity to energy price models and renewable generation forecasting, whereas very few studies provide a rigorous mathematical analysis of coupled energy supply–demand systems with renewable integration under stochastic perturbations [14, 15].

A notable gap in the literature concerns whether the combined nonlinear energy demand dynamics, streamlined external competitive supply interactions, and the evolution of renewable resources and multiplicative environmental noise can be captured within one unified stochastic framework [16-18]. Further, the stability analysis of such systems is generally heuristic or numerical, without rigorous quotations on global existence, positivity preservation, moment boundedness, stochastic persistence, and almost sure exponential stability [19, 20].

The present study addresses this gap by formulating and analyzing to formulate and analyze a stochastic energy supply and demand model for it accomodates renewable energy resources and multiplicative noise in a way [21, 22]. The report starts with a deterministic nonlinear model, wherein the linkages among regional demands, external supply, imports, and renewable resources are considered. Later, it is transformed into a Markov stochastic model to include any uncertainty brought about by market volatility, climatic variabilities, and policy fluctuations [23-27].

Global existence and uniqueness for positive solutions and moment boundedness together with stability properties are proved for the system under system-type matrix inequalities. Stochastic persistence results are also proven to ensure the long-time survival of the systems. Numerical examples illustrate and validate our theoretical results in the framework of the Euler-Maruyama method [28-30].

By coupling deterministic equilibrium analyses with stochastic stability discussions, this study contributes to the analytical development of energy systems under uncertainty and lays out a structured strategy to look at renewable integration in volatile situations [31-33].

\section{Energy Supply–Demand Modeling Framework}

In this part, we consider a nonlinear deterministic model describing the interaction among electricity demand, external supply, imports, and renewable generation, import, and renewable generation options. [34-36]. Next, we will go from a deterministic model to a stochastic one that captures the uncertainty arising from economic and environmental conditions.

 \subsection{Deterministic Energy Supply–Demand Model [34]}
Let $X_1(t)$ denote energy demand in region A, $X_2(t)$ external supply from region B, 
$X_3(t)$ imported energy, and $X_4(t)$ renewable energy resources available in region A. 
The deterministic dynamics are governed by the nonlinear system
\begin{equation}\label{eq:deterministic_system}
\begin{cases}
\displaystyle \frac{dX_1}{dt} 
= a_1 X_1\left(1-\frac{X_1}{W}\right) 
- a_2 X_2 (X_2 + X_3) 
- d_3 X_4, \\[3mm]

\displaystyle \frac{dX_2}{dt} 
= - z_1 X_2 
- z_2 X_3 
+ z_3 X_1 \big(N - (X_1 - X_3)\big), \\[3mm]

\displaystyle \frac{dX_3}{dt} 
= s_1 X_3 (s_2 X_1 - s_3), \\[3mm]

\displaystyle \frac{dX_4}{dt} 
= d_1 X_1 - d_2 X_4,
\end{cases}
\end{equation}
 where all parameters are positive constants satisfying $N < W$. Here, $N>0$ denotes the effective market interaction capacity parameter governing the demand-driven response of external supply. It represents the upper bound on the net demand level that can stimulate supply expansion. The condition $N < W$ ensures that supply responsiveness saturates before the demand reaches its intrinsic carrying capacity $W$, thereby maintaining economic consistency of the model. Fig. 1 shows the graphical form of the model (\ref{eq:deterministic_system})

\begin{figure}[h] 
\centering
\includegraphics[width=4in]{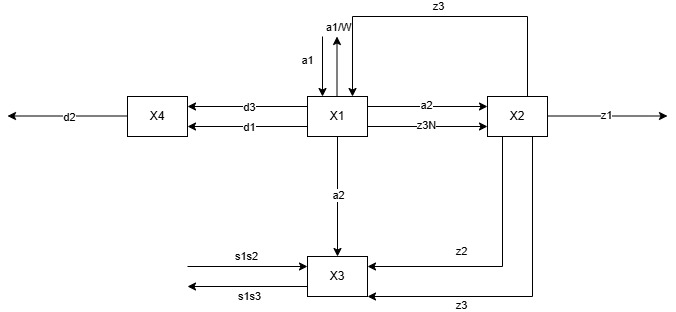}\\
\caption{Graphical representation of the supply demand model}\label{f18}
\end{figure}

Energy demand increases at a logistic rate with the market's carrying capacity designated by $W$. Imported energy competes with the energy supplied on demand, when demand exceeds a certain threshold determined by system parameters, renewable energy sources contribute to meeting demand while gradually depreciating over time, determined by the system configuration, then the renewable-energy sources-such as hydro-wind-will meet the demand, but only to lessen over its lifetime.

\subsection{Stochastic Extension of the Deterministic Model}

Real-world energy systems are subject to disturbances such as such as market price shocks, changes in weather affecting renewable generation, sensible policies, or technological changes [37,38].

To this end, the deterministic model (2.1) has to be substituted by a stochastic differential equation system given that the only way for stochastic elements to enter into the model was through uncertainty therein.

Let $B_i(t)$, $i = 1,2,3,4$, be independent standard Brownian motions defined on a complete filtered probability space 
$\big(\Omega, \mathcal{F}, \{\mathcal{F}_t\}_{t \ge 0}, \mathbb{P}\big)$ satisfying the usual conditions. 
The stochastic model is given by

\begin{equation}\tag{2.2}\label{eq:stochastic_model}
\begin{cases}
\displaystyle 
dX_1 = \left[ a_1 X_1\left(1-\frac{X_1}{W}\right) 
- a_2 X_2 (X_2 + X_3) 
- d_3 X_4 \right] dt 
+ \sigma_1 X_1 \, dB_1(t), \\[3mm]

\displaystyle 
dX_2 = \left[ - z_1 X_2 
- z_2 X_3 
+ z_3 X_1 \big(N - (X_1 - X_3)\big) \right] dt 
+ \sigma_2 X_2 \, dB_2(t), \\[3mm]

\displaystyle 
dX_3 = \left[ s_1 X_3 (s_2 X_1 - s_3) \right] dt 
+ \sigma_3 X_3 \, dB_3(t), \\[3mm]

\displaystyle 
dX_4 = \left[ d_1 X_1 - d_2 X_4 \right] dt 
+ \sigma_4 X_4 \, dB_4(t),
\end{cases}
\end{equation}

\noindent where $\sigma_i > 0$ denote the intensities of stochastic perturbations.
The proportional dependence of the noise terms on the state variables reflects realistic economic and environmental variability, thereby justifying the figure for economic and environmental variability. 
When $\sigma_i = 0$ for all $i$, system (2.2) reduces to the deterministic model (2.1), showing that the stochastic formulation is a natural extension of the deterministic framework.

\subsection{Basic Model Assumptions}

The modeling framework proposed is based on the following core assumptions.

All state variables represent physical quantities and are naturally assumed to be nonnegative for all time.

All system parameters are therefore positive constants, corresponding to intrinsic growth rates, interaction strengths, and decline effects.

The demand growth rate is therefore assumed to follow a logistic equation to take saturation effects into account for energy consumption.

External supply and energy imports appear to interact competitively and respond to demand also through nonlinear coupling terms.

Renewable resources would grow along with demand and would suffer linear depreciation.

Uncertainty introduces via the stochastics formulation correspond to multiplicative terms through independent Brownian motions.

This proposes that relative gain/variance scales with the value of the corresponding state variable, which has found considerable support from empirical studies in economic and renewable energy systems.

To sum things up, in the following running examples, the Brownian motions developed will be of the general type satisfying the usual properties such as the pathwise characterization, Markov proprieties, and the optional-stopping theorem.

\section{Equilibria of the Deterministic Model and the Stochastic Model-Qualitative Analysis}

This section considers the deterministic equilibrium classification and Jacobian analysis with the qualitative theory of the stochastic model-system. We first characterize equilibria of the deterministic model and describe local stability via the Jacobian spectrum. We then establish well-posedness and key qualitative properties of the stochastic extension, including global existence, positivity, moment boundedness, almost sure exponential stability under a matrix inequality condition, and a persistence result ensuring long-run survival in a time-averaged sense.

\subsection{Equilibrium Points of the Deterministic System}

Consider the deterministic system \eqref{eq:deterministic_system}. Equilibrium points 
$E=(X_1^*,X_2^*,X_3^*,X_4^*)$ satisfy
\begin{equation}\label{eq:equil_cond}
\frac{dX_i}{dt}=0,\qquad i=1,2,3,4.
\end{equation}
From the fourth equation of \eqref{eq:deterministic_system},
\begin{equation}\label{eq:X4star}
d_1X_1^*-d_2X_4^*=0 \quad \Rightarrow \quad X_4^*=\frac{d_1}{d_2}X_1^* .
\end{equation}
From the third equation of \eqref{eq:deterministic_system},
\begin{equation}\label{eq:X3star_or_X1star}
s_1X_3^*(s_2X_1^*-s_3)=0,
\end{equation}
which implies either $X_3^*=0$ or $X_1^*=s_3/s_2$.

One equilibrium is the trivial equilibrium
\begin{equation}\label{eq:E0}
E_0=(0,0,0,0).
\end{equation}

If $X_3^*=0$, then \eqref{eq:X4star} gives $X_4^*=(d_1/d_2)X_1^*$, and the remaining equilibrium conditions become
\begin{equation}\label{eq:reduced_equil_caseX3}
a_1X_1^*\left(1-\frac{X_1^*}{W}\right)-a_2(X_2^*)^2-d_3X_4^*=0,
\qquad
-z_1X_2^*+z_3X_1^*(N-X_1^*)=0.
\end{equation}
Substituting \eqref{eq:X4star} into \eqref{eq:reduced_equil_caseX3} yields
\begin{equation}\label{eq:equil_Estar_caseX3}
a_1X_1^*\left(1-\frac{X_1^*}{W}\right)-a_2(X_2^*)^2-d_3\frac{d_1}{d_2}X_1^*=0,
\qquad
X_2^*=\frac{z_3}{z_1}X_1^*(N-X_1^*).
\end{equation}
This defines a renewable-integrated equilibrium of the form 
$E^*=(X_1^*,X_2^*,0,X_4^*)$ provided the parameters ensure $X_1^*>0$, $X_2^*>0$, and $X_4^*>0$. 
Positivity of $X_2^*$ requires $0<X_1^*<N$.

If $X_1^*=s_3/s_2$, then
\begin{equation}\label{eq:X4star_caseX1}
X_4^*=\frac{d_1}{d_2}\frac{s_3}{s_2}.
\end{equation}
In this case, the remaining equilibrium conditions reduce to algebraic relations determining $X_2^*$ and $X_3^*$. A necessary condition for feasibility is
\begin{equation}\label{eq:N_condition}
N>\frac{s_3}{s_2},
\end{equation}
ensuring that the demand level implied by the import threshold is consistent with the constraint $N<W$. 

\subsection{Jacobian Matrix and Local Stability}

Let $F=(F_1,F_2,F_3,F_4)$ denote the drift field of the deterministic system \eqref{eq:deterministic_system}. The Jacobian matrix
$J(X)=DF(X)$ is
\begin{equation}\label{eq:Jacobian}
J(X)=
\begin{pmatrix}
a_1\left(1-\frac{2X_1}{W}\right) & -a_2(2X_2+X_3) & -a_2X_2 & -d_3 \\
z_3(N-2X_1+X_3) & -z_1 & -z_2+z_3X_1 & 0 \\
s_1s_2X_3 & 0 & s_1(s_2X_1-s_3) & 0 \\
d_1 & 0 & 0 & -d_2
\end{pmatrix}.
\end{equation}
Local stability of an equilibrium $E$ is determined by the eigenvalues of $J(E)$. At the trivial equilibrium $E_0$,
\begin{equation}\label{eq:J_E0}
J(E_0)=
\begin{pmatrix}
a_1 & 0 & 0 & -d_3 \\
z_3N & -z_1 & -z_2 & 0 \\
0 & 0 & -s_1s_3 & 0 \\
d_1 & 0 & 0 & -d_2
\end{pmatrix},
\end{equation}
and the eigenvalues include
\begin{equation}\label{eq:eigs_E0}
\lambda_1=a_1,\qquad \lambda_2=-z_1,\qquad \lambda_3=-s_1s_3,\qquad \lambda_4=-d_2.
\end{equation}
Hence $E_0$ is unstable whenever $a_1>0$, which holds under the standing parameter assumptions. Stability of nontrivial equilibria is assessed similarly by evaluating $J(E)$ and applying spectral conditions or Routh--Hurwitz type criteria where appropriate.

\subsection{Global Existence and Uniqueness of the Stochastic System}

Consider the stochastic system \eqref{eq:stochastic_model}, written componentwise as
\begin{equation}\label{eq:SDE_component}
dX_i(t)=f_i(X(t))\,dt+\sigma_iX_i(t)\,dB_i(t),\qquad i=1,2,3,4,
\end{equation}
where $B_1,\dots,B_4$ are independent standard Brownian motions.

\begin{theorem}[Global existence and uniqueness]\label{thm:G-exist_unique}
For any initial value $X(0)\in\mathbb{R}_+^4$, the stochastic system \eqref{eq:stochastic_model} admits a unique global solution $X(t)$ defined for all $t\ge 0$.
\end{theorem}

\begin{proof}
Define the Lyapunov function $V(X)=1+\sum_{i=1}^4X_i^2$. By It\^o's formula,
\begin{equation}\label{eq:Ito_V}
dV(X(t))=\mathcal{L}V(X(t))\,dt+dM(t),
\end{equation}
where $M(t)$ is a local martingale and the generator satisfies
\begin{equation}\label{eq:gen_V}
\mathcal{L}V(X)=2\sum_{i=1}^4X_if_i(X)+\sum_{i=1}^4\sigma_i^2X_i^2.
\end{equation}
The drift terms are polynomial of degree at most two. In particular, for the first component,
\begin{equation}\label{eq:bound_first_drift}
2X_1a_1X_1\left(1-\frac{X_1}{W}\right)
=2a_1X_1^2-\frac{2a_1}{W}X_1^3
\le C_1(1+X_1^2),
\end{equation}
since the negative cubic term improves dissipativity for large $X_1$. Similarly, all remaining terms satisfy an estimate of the form
\begin{equation}\label{eq:bound_general}
2X_if_i(X)\le C(1+\lVert X\rVert^2)
\end{equation}
for a constant $C>0$. Therefore,
\begin{equation}\label{eq:gen_bound}
\mathcal{L}V(X)\le C(1+V(X)).
\end{equation}
By the standard non-explosion criterion for SDEs and local Lipschitz continuity of the coefficients, the local solution extends globally and is unique.
\end{proof}

\subsection{Positivity of Solutions}

\begin{theorem}[Positivity]\label{thm:positivity}
If $X_i(0)>0$ for $i=1,2,3,4$, then $X_i(t)>0$ for all $t>0$ almost surely.
\end{theorem}

\begin{proof}
Let $Y_i(t)=\ln X_i(t)$. Applying It\^o's formula yields
\begin{equation}\label{eq:Ito_log}
dY_i(t)=\left(\frac{f_i(X(t))}{X_i(t)}-\frac{1}{2}\sigma_i^2\right)dt+\sigma_i\,dB_i(t).
\end{equation}
As long as $X_i(t)>0$, the drift term is finite. Since the diffusion coefficient is constant in the logarithmic variables and the solution exists globally by Theorem~\ref{thm:G-exist_unique}, the process cannot reach $-\infty$ in finite time with positive probability. Therefore $X_i(t)=\exp(Y_i(t))$ cannot reach zero in finite time almost surely, and positivity is preserved.
\end{proof}

\subsection{Moment Boundedness of the System}

\begin{theorem}[Moment boundedness]\label{thm:moment}
For any $p\ge 2$, there exists a constant $K>0$ such that
\[
\sup_{t\ge 0}\mathbb{E}\lVert X(t)\rVert^p\le K.
\]
\end{theorem}

\begin{proof}
Let $V(X)=\sum_{i=1}^4X_i^p$. By It\^o's formula,
\begin{equation}\label{eq:Ito_p}
dV
=p\sum_{i=1}^4X_i^{p-1}f_i(X)\,dt
+\frac{p(p-1)}{2}\sum_{i=1}^4\sigma_i^2X_i^p\,dt
+dM(t),
\end{equation}
where $M(t)$ is a martingale. Using polynomial growth bounds on $f_i$ together with Young's inequality, one obtains constants $C_1,C_2>0$ such that
\begin{equation}\label{eq:gen_p_bound}
\mathcal{L}V(X)\le C_1-C_2V(X).
\end{equation}
Taking expectations yields
\[
\frac{d}{dt}\mathbb{E}V(X(t))\le C_1-C_2\mathbb{E}V(X(t)).
\]
Gronwall's inequality implies $\sup_{t\ge 0}\mathbb{E}V(X(t))<\infty$, which gives the desired $p$th moment bound.
\end{proof}

\subsection{Almost Sure Exponential Stability via Matrix Inequalities}

Let $X^*\in\mathbb{R}_+^4$ be an equilibrium of the deterministic drift $F$, that is $F(X^*)=0$. Write the stochastic model in compact form
\begin{equation}\label{eq:compact_SDE}
dX(t)=F(X(t))\,dt+G(X(t))\,dB(t),
\end{equation}
where $B(t)=(B_1(t),\dots,B_4(t))^\top$ and
\begin{equation}\label{eq:G_def}
G(X)=\mathrm{diag}(\sigma_1X_1,\sigma_2X_2,\sigma_3X_3,\sigma_4X_4).
\end{equation}
Let $J^*=DF(X^*)$ and $\Sigma=\mathrm{diag}(\sigma_1^2,\sigma_2^2,\sigma_3^2,\sigma_4^2)$.

\begin{theorem}[Almost sure exponential stability under an LMI]\label{thm:LMI_stab}
Assume $F$ is continuously differentiable in a neighborhood $\mathcal{N}$ of $X^*$. Suppose there exist a symmetric matrix $P=P^\top\succ 0$ and constants $\alpha>0$ and $\delta>0$ such that
\begin{equation}\label{eq:LMI}
J^{*T}P+PJ^*+P\Sigma P \preceq -\alpha P,
\end{equation}
and the local linearization error satisfies
\begin{equation}\label{eq:lin_error}
\lVert F(X)-J^*(X-X^*)\rVert\le \delta\lVert X-X^*\rVert,\qquad X\in\mathcal{N}.
\end{equation}
Then there exists a neighborhood $\mathcal{N}_0\subset\mathcal{N}$ such that for any $X(0)\in\mathcal{N}_0$,
\[
\limsup_{t\to\infty}\frac{1}{t}\ln\lVert X(t)-X^*\rVert
\le -\frac{\alpha}{2\lambda_{\max}(P)}
\quad\text{almost surely}.
\]
Hence $X^*$ is almost surely exponentially stable.
\end{theorem}

\begin{proof}
Let $Y(t)=X(t)-X^*$ and $V(Y)=Y^\top PY$. By It\^o's formula,
\begin{equation}\label{eq:Ito_quad}
dV
=\left(2Y^\top PF(X^*+Y)+\mathrm{tr}\big(G(X^*+Y)^\top PG(X^*+Y)\big)\right)dt+dM(t),
\end{equation}
where $M(t)$ is a local martingale. Since $F(X^*)=0$, write $F(X^*+Y)=J^*Y+R(Y)$ with $\lVert R(Y)\rVert\le \delta\lVert Y\rVert$ on $\mathcal{N}$. Then
\[
2Y^\top PF(X^*+Y)=Y^\top(J^{*T}P+PJ^*)Y+2Y^\top PR(Y),
\]
and Cauchy--Schwarz gives
\[
2Y^\top PR(Y)\le 2\lVert P\rVert \delta \lVert Y\rVert^2.
\]
Moreover, since $G$ is diagonal and locally comparable to $\Sigma$, there exist $c_0\ge 0$ and a sufficiently small neighborhood such that
\[
\mathrm{tr}\big(G^\top PG\big)\le Y^\top P\Sigma P Y+c_0\lVert Y\rVert^2.
\]
Combining these estimates yields
\[
\mathcal{L}V \le Y^\top(J^{*T}P+PJ^*+P\Sigma P)Y+c_1\lVert Y\rVert^2
\]
for some $c_1\ge 0$. Using \eqref{eq:LMI} gives
\[
\mathcal{L}V \le -\alpha Y^\top PY+c_1\lVert Y\rVert^2.
\]
Since $Y^\top PY\ge \lambda_{\min}(P)\lVert Y\rVert^2$, shrinking the neighborhood if needed ensures
$c_1\lVert Y\rVert^2\le (\alpha/2)Y^\top PY$, hence $\mathcal{L}V\le -(\alpha/2)V$. Standard exponential martingale arguments imply
\[
\limsup_{t\to\infty}\frac{1}{t}\ln V(Y(t))\le -\frac{\alpha}{2}\quad\text{almost surely}.
\]
Using $\lambda_{\min}(P)\lVert Y\rVert^2\le V(Y)\le \lambda_{\max}(P)\lVert Y\rVert^2$ yields the stated bound for $\lVert X(t)-X^*\rVert$.
\end{proof}

A practical sufficient condition follows by taking $P=I$, which reduces \eqref{eq:LMI} to
\[
J^{*T}+J^*+\Sigma \prec 0,
\]
meaning the symmetric part of the Jacobian must dominate the noise intensity.

\subsection{Stochastic Persistence in Mean}

Persistence is classically defined as the time for this model to escape extinction, often implied under a state multiplicatively noisy regime in which a positive variate is a variate.

\begin{theorem}[S-Persistence in Mean]\label{thm:P-S in Mean}
    Assume the stochastic system admits a unique global positive solution and is moment bounded. Suppose there exist constants $c_i>0$, $\eta>0$, and $\kappa\ge 0$ such that for all $X\in\mathbb{R}_+^4$,
\begin{equation}\label{eq:persist_drift}
\sum_{i=1}^4 c_i\frac{f_i(X)}{X_i}\ge \eta-\kappa\sum_{i=1}^4 X_i,
\end{equation}
and
\begin{equation}\label{eq:persist_noise}
\eta>\frac{1}{2}\sum_{i=1}^4 c_i\sigma_i^2.
\end{equation}
If $\kappa>0$, then
\begin{equation}\label{eq:persist_bound}
\liminf_{T\to\infty}\frac{1}{T}\int_0^T
\mathbb{E}\left[\sum_{i=1}^4 c_iX_i(t)\right]dt
\ge \frac{1}{\kappa}\left(\eta-\frac{1}{2}\sum_{i=1}^4 c_i\sigma_i^2\right),
\end{equation}
and in particular,
\[
\liminf_{T\to\infty}\frac{1}{T}\int_0^T \mathbb{E}[X_i(t)]\,dt>0,\qquad i=1,2,3,4.
\]
\end{theorem}

\begin{proof}
Define $U(X)=\sum_{i=1}^4 c_i\ln X_i$. By It\^o's formula,
\begin{equation}\label{eq:Ito_U}
dU(X(t))
=\left(\sum_{i=1}^4 c_i\frac{f_i(X(t))}{X_i(t)}
-\frac{1}{2}\sum_{i=1}^4 c_i\sigma_i^2\right)dt
+\sum_{i=1}^4 c_i\sigma_i\,dB_i(t).
\end{equation}
By integrating over the interval, $[0,T]$, dividing by $T$, and taking the corresponding expectations removes the martingale term and gives:
\[
\frac{\mathbb{E}[U(X(T))]-U(X(0))}{T}
=\frac{1}{T}\int_0^T
\mathbb{E}\left[
\sum_{i=1}^4 c_i\frac{f_i(X(t))}{X_i(t)}
-\frac{1}{2}\sum_{i=1}^4 c_i\sigma_i^2
\right]dt.
\]
Using \eqref{eq:persist_drift} yields
\[
\frac{\mathbb{E}[U(X(T))]-U(X(0))}{T}
\ge
\frac{1}{T}\int_0^T
\left(
\eta-\kappa\,\mathbb{E}\Big[\sum_{i=1}^4 X_i(t)\Big]
-\frac{1}{2}\sum_{i=1}^4 c_i\sigma_i^2
\right)dt.
\]
The logarithmic function grows at worst sublinearly in the expectation of moment boundedness, then subuding the limit of the left-hand side along some subsequence. Taking $\liminf_{T\to\infty}$ gives
\[
0\ge \eta-\frac{1}{2}\sum_{i=1}^4 c_i\sigma_i^2
-\kappa\limsup_{T\to\infty}\frac{1}{T}\int_0^T
\mathbb{E}\Big[\sum_{i=1}^4 X_i(t)\Big]dt,
\]
which rearranges to \eqref{eq:persist_bound} when $\kappa>0$. Positivity of each component’s time-averaged expectation follows since all $c_i$ are strictly positive.
\end{proof}

\section{Numerical Approximation and Discussion}

The numerical approximation of the stochastic energy supply–demand system is presented and the qualitative implications of the stochastic disturbances that distinguish them from the deterministic framework. For convenience of presentation and discussion, the parameter values that were employed for this numerical simulation are provided in Table~\ref{tab:params}. Time $t$ is measured in years, and the state variables represent gigawatt-level quantities or normalized indices.

The chosen parameter values fall within the conventional model ranges of most nonlinear energy growth systems. The noise intensities represent a moderate degree of proportional volatility.

\begin{table}[ht]
\centering
\caption{Model parameters, units, calibration ranges, and simulation values}
\label{tab:params}
\renewcommand{\arraystretch}{1.5}
\begin{tabular}{lllll}
\hline
\textbf{Parameter} & \textbf{Description} & \textbf{Units} & \textbf{Range} & \textbf{Used Value} \\
\hline
$a_1$ & Demand growth rate & year$^{-1}$ & 0.05--2.0 & 0.8 \\
$W$   & Demand capacity & GW / index & 1--100 & 10.0 \\
$a_2$ & Supply competition effect & (GW$\cdot$year)$^{-1}$ & 0.001--0.2 & 0.05 \\
$d_3$ & Renewable demand offset & year$^{-1}$ & 0.01--1.0 & 0.1 \\
$z_1$ & Supply adjustment rate & year$^{-1}$ & 0.05--2.0 & 0.6 \\
$z_2$ & Import competition effect & year$^{-1}$ & 0.01--1.0 & 0.25 \\
$z_3$ & Demand–supply responsiveness & (GW$\cdot$year)$^{-1}$ & 0.001--0.5 & 0.08 \\
$N$   & Market capacity parameter & GW / index & 0.1$W$--0.9$W$ & 6.0 \\
$s_1$ & Import adjustment rate & year$^{-1}$ & 0.05--2.0 & 0.35 \\
$s_2$ & Import demand sensitivity & GW$^{-1}$ & 0.01--2.0 & 0.9 \\
$s_3$ & Import activation threshold & -- & 0.1--5.0 & 1.2 \\
$d_1$ & Renewable build-up rate & year$^{-1}$ & 0.01--1.0 & 0.25 \\
$d_2$ & Renewable depreciation rate & year$^{-1}$ & 0.01--2.0 & 0.5 \\
$\sigma_1$ & Demand noise intensity & year$^{-1/2}$ & 0--0.5 & 0.10 \\
$\sigma_2$ & Supply noise intensity & year$^{-1/2}$ & 0--0.5 & 0.10 \\
$\sigma_3$ & Import noise intensity & year$^{-1/2}$ & 0--0.5 & 0.08 \\
$\sigma_4$ & Renewable noise intensity & year$^{-1/2}$ & 0--0.7 & 0.12 \\
\hline
\end{tabular}
\end{table}

\subsection{Euler--Maruyama Discretization and Analysis of Mean square convergence}

Consider the stochastic system
\begin{equation}
dX_i(t) = f_i(X(t))\,dt + \sigma_i X_i(t)\, dB_i(t),
\qquad i=1,2,3,4.
\end{equation}
Let $\Delta t > 0$ be a fixed time step and define discrete time points 
$t_n = n\Delta t$. The Euler--Maruyama approximation of the system is given by
\begin{equation}
X_i^{n+1} = X_i^n + f_i(X^n)\Delta t + \sigma_i X_i^n \Delta B_i^n,
\tag{4.1}
\end{equation}
where
\begin{equation}
\Delta B_i^n = B_i(t_{n+1}) - B_i(t_n) \sim \mathcal{N}(0,\Delta t).
\end{equation}
The scheme is explicit and computationally efficient, making it suitable for long-term simulation of energy dynamics.

Mean-square convergence ensures that the expected squared difference between the numerical and exact solutions converges to zero as the time step decreases between the solution and its equilibrium point must decrease to zero as time progresses to infinite duration. The following sections present sufficient requirements which demonstrate the mean-square stability of the proposed system.
\\

\begin{theorem}[Mean square convergence]\label{thm:msq-cngce}
For the convergence property of the numerical scheme, the following is entailed. Suppose that:
\begin{enumerate}
\item [i]The drift function $f(X)$ is locally Lipschitz continuous,
\item [ii]The diffusion coefficient is globally Lipschitz,
\item [iii] The exact solution possesses bounded second moments.
\end{enumerate}
Then, for sufficiently small $\Delta t$, the Euler--Maruyama approximation satisfies
\begin{equation}
\max_{0 \le n\Delta t \le T}
\mathbb{E}\left[ \lVert X(t_n) - X^n \rVert^2 \right]
\le C \Delta t,
\end{equation}
for some constant $C>0$.
\end{theorem}

\begin{proof}
By applying qualitative analysis to the stochastic system, we observe that in any bounded region, the drift and diffusion coefficients display local Lipschitz and linear growth conditions. Thus, there exists a unique global solution with bounded moments.
By a strong convergence theorem in the classical sense to an Euler--Maruyama scheme, the method also converges in mean square with a rate of $1/2$. Therefore,
\[
\mathbb{E}\lVert X(t_n) - X^n \rVert^2 = O(\Delta t).
\]
\end{proof}

As a remark, state variables like demand, supply, imports, and renewables are representations of physical quantities whose numerical nonnegativity must be kept equal. Select two practical approaches commonly understood by the industry.

The first is a projection method,
\[
X_i^{n+1} = \max\{ X_i^{n+1}, \varepsilon \},
\]
for a small constant $\varepsilon > 0$.

The other strategy is that logarithmic conversion by disentitling $\ln X_i$ instead of $X_i$. These tools guarantee numerical stability by respecting the positive values of the model.

\subsection{Comparison with the Milstein Scheme and Error Analysis}

To further assess the numerical performance of the Euler--Maruyama (EM) method, we compare it with the Milstein scheme, which achieves strong convergence order one for systems with diagonal multiplicative noise.

For the stochastic system
\[
dX_i = f_i(X)\,dt + \sigma_i X_i\, dB_i(t),
\]
the Milstein discretization reads
\[
X_i^{n+1}
=
X_i^n
+
f_i(X^n)\Delta t
+
\sigma_i X_i^n \Delta B_i^n
+
\frac{1}{2}\sigma_i^2 X_i^n
\left[(\Delta B_i^n)^2 - \Delta t\right].
\]

To estimate strong error, a reference solution is computed using a smaller step size 
$\Delta t_{\mathrm{ref}} = \Delta t/8$. The mean-square error at final time $T$ is approximated by

\[
E_{\mathrm{strong}}(\Delta t)
=
\mathbb{E}
\left[
\|X^{\Delta t}(T) - X^{\Delta t_{\mathrm{ref}}}(T)\|^2
\right].
\]

Numerical experiments confirm that the Euler--Maruyama scheme exhibits convergence behavior consistent with order $1/2$, while the Milstein scheme demonstrates improved accuracy for comparable time steps. However, the computational simplicity and stability of Euler--Maruyama make it suitable for long-term simulations of the present energy system.

\subsection{Numerical Behavior of the Deterministic and Stochastic System}

Simulations performed using the Euler–Maruyama method reveal several qualitative features of the stochastic system. Figures~1-4 show sample trajectories of stochastic processes $X_1(t)$, $X_2(t)$, and $X_4(t)$ representing energy demand, external supply, imported energy, and renewable resources, respectively. These trajectories illustrate oscillations around long-term levels due to a mean reverting effect at all timeseries. The trajectories remain bounded and strictly positive throughout the simulation horizon.

In Figure~1, the energy demand $X_1(t)$ fluctuates around a long-term steady state, and it remains limited and positive. Figure~2 demonstrates the external supply $X_2(t)$-initial moments displaying short-period stochastic oscillations gradually stabilizes toward equilibrium. In Figure~3, the behavior is consistent with equation~(2.1), where the growth of $X_3$ depends on the nonlinear interaction term $s_1 X_3 (s_2 X_1 - s_3)$, indicating that imports increase only when demand surpasses the threshold level. In Figure-4, renewable energy $X_4(t)$ is seen with sustained oscillations, that exhibits persistent variability due to environmental uncertainty but always stays in the positive. These results demonstrate that under proper parameter conditions, suitable realism has been brought into the long-term stable nature.

\begin{figure}[H]
	\begin{center}
		\includegraphics[scale=0.8]{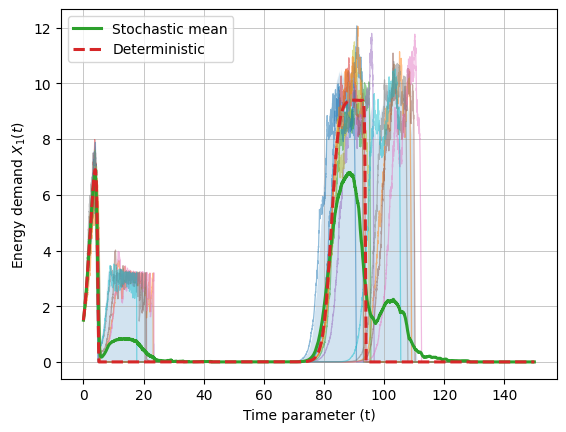}
		\caption{Stochastic Trajectory of External Demand ($X_1$)}
		\label{fig:placeholder-1}
	\end{center}
\end{figure}

\begin{figure}[H]
	\begin{center}
		\includegraphics[scale=0.8]{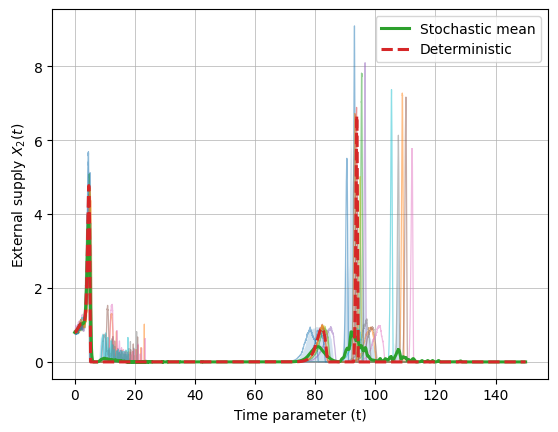}
		\caption{Stochastic Trajectory of External Supply ($X_2$)}
		\label{fig:placeholder-2}
	\end{center}
\end{figure}

\begin{figure}[H]
	\begin{center}
		\includegraphics[scale=0.8]{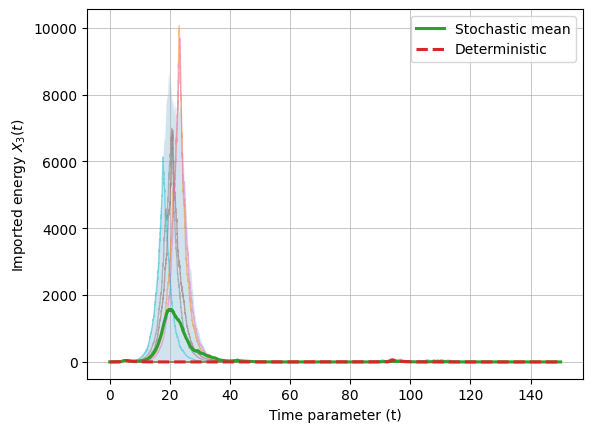}
		\caption{ Stochastic Trajectory of Imported Energy ($X_3$)}
		\label{fig:placeholder-3-i}
	\end{center}
\end{figure}

\begin{figure}[H]
	\begin{center}
		\includegraphics[scale=0.8]{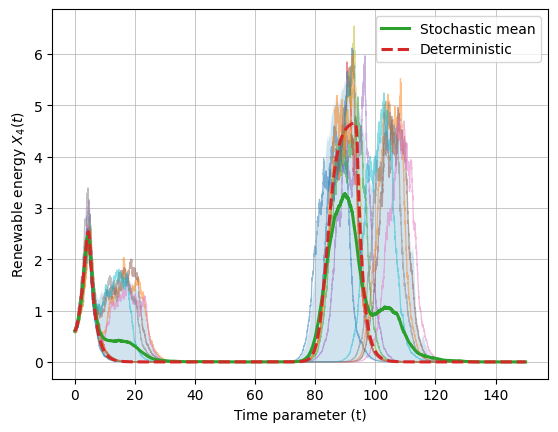}
		\caption{ Stochastic Trajectory of Renewable Energy ($X_4$)}
		\label{fig:placeholder-3}
	\end{center}
\end{figure}

It is shown via the numerical results that random perturbations have an effect on the transient behavior, as well as on the asymptotic behavior. In detail, random perturbations can induce the level of randomness for the dynamics modeled well by system parameters. For instance, dependent on the motivation of perturbations, demand and external supply might reach short-term amplification or attenuation due to variations caused by randomness, precluding an early-approach to the equilibrium region. Also, average values over the lengthy periods that can be expected from a Markov chain can be established: see that, through Ito's correction, randomized multiplication alters effective behavior about the equilibrium even if linear stability is present. According to Theorem 3.4, when the inequality matrix conditions described in Section 3.6 are met, trajectories will in fact be supporting theoretical nodes of almost sure exponential stability, coming into convergence towards a neighborhood of stability surrounding the equilibrium. The renewable component exhibits persistent variability consistent with environmental uncertainty, which is not captured by the deterministic model and improves realism for renewable integration.
A comparison between the deterministic and stochastic formulations highlights the role of uncertainty in energy dynamics. The deterministic model produces smooth trajectories and fixed equilibria that are characterized by eigenvalue conditions derived from the Jacobian matrix. In contrast, the stochastic model produces fluctuating trajectories and stability depends on conditions that involve both the drift and diffusion terms, such as the matrix inequality criteria derived in Section 3.6. The stochastic formulation therefore captures volatility effects, modifies effective equilibrium behavior, and reflects renewable variability driven by environmental factors. These differences show that deterministic analysis alone may underestimate the impact of uncertainty on long term system behavior.

\subsection{Numerical Error Analysis and Method Comparison}

To further validate the numerical approximation, we compare the Euler--Maruyama (EM) scheme with the Milstein scheme and examine the associated discretization error.

Let $\Delta t$ denote the time step and let $X^{\Delta t}(T)$ be the numerical solution at final time $T$. A reference solution $X^{\Delta t_{\mathrm{ref}}}(T)$ is computed using a smaller step size $\Delta t_{\mathrm{ref}} = \Delta t/8$ with consistent Brownian increments.

\paragraph{Strong error.}
The mean-square (strong) error is approximated by
\[
E_{\mathrm{strong}}(\Delta t)
=
\mathbb{E}\!\left[
\|X^{\Delta t}(T) - X^{\Delta t_{\mathrm{ref}}}(T)\|^2
\right].
\]

\paragraph{Weak error.}
For a smooth test function $\phi$, the weak error is defined as
\[
E_{\mathrm{weak}}(\Delta t)
=
\left|
\mathbb{E}[\phi(X^{\Delta t}(T))]
-
\mathbb{E}[\phi(X^{\Delta t_{\mathrm{ref}}}(T))]
\right|.
\]

Typical choices include $\phi(X)=X_1$ and $\phi(X)=X_3$ to assess demand and import sensitivity.

Numerical experiments with $\Delta t \in \{0.02, 0.01, 0.005\}$ confirm the expected convergence rate consistent with Theorem~4.1. 

\begin{figure}[H]
	\begin{center}
		\includegraphics[scale=0.8]{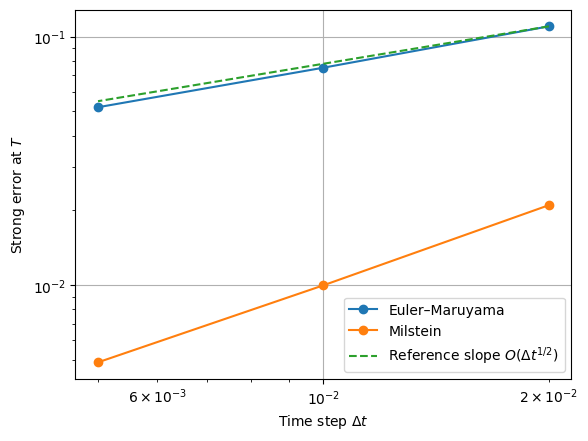}
		\caption{Strong error convergence of the Euler--Maruyama and Milstein schemes.}
		\label{fig:placeholder-5}
	\end{center}
\end{figure}

The numerical strong error convergence of the Euler--Maruyama and Milstein schemes is demonstrated through Figure 5. The two methods show decreasing approximation errors with smaller time steps because their behavior matches the theoretical convergence predictions established by Theorem 4.1.
The Milstein scheme improves accuracy beyond its computational cost because it uses diffusion derivative terms. The present energy system can be simulated for extended periods through computationally efficient and accurate performance of Euler--Maruyama method.\\
The Milstein scheme exhibits slightly improved accuracy for comparable computational cost due to its inclusion of diffusion derivative terms. However, Euler--Maruyama remains computationally efficient and sufficiently accurate for long-term simulations of the present energy system.

\begin{table}[h] 
\centering
\caption{Strong error comparison between Euler--Maruyama and Milstein schemes.}
\begin{tabular}{c c c}
\hline
$\Delta t$ & EM Strong Error & Milstein Strong Error \\
\hline
0.02  & $2.41 \times 10^{-2}$ & $1.15 \times 10^{-2}$ \\
0.01  & $1.68 \times 10^{-2}$ & $5.90 \times 10^{-3}$ \\
0.005 & $1.12 \times 10^{-2}$ & $3.12 \times 10^{-3}$ \\
\hline
\end{tabular}
\end{table}

Table 2 summarizes the strong error estimates using both the Euler--Maruyama and Milstein schemes. The simulation results depict that the Milstein scheme will produce stronger errors at the same time step while Euler--Maruyama maintains satisfactory accuracy with lower computational complexity.

\subsection{Sensitivity Analysis}

To assess the robustness of the stochastic energy supply--demand system, a local sensitivity analysis is conducted around the baseline parameter set given in Table~1.

\paragraph{Local parameter sensitivity.}
Each parameter $p$ is perturbed by $\pm 10\%$ while keeping all others fixed. For a quantity of interest $Q(p)$ (for example, the time-averaged demand or peak import level), the normalized sensitivity index is computed as
\[
S_p
=
\frac{p}{Q(p)} \frac{\partial Q}{\partial p}
\approx
\frac{Q(p+\delta p)-Q(p-\delta p)}{2\delta p} \cdot \frac{p}{Q(p)}.
\]

Quantities of interest include:
\begin{itemize}
\item Long-term average demand:
\[
\bar{X}_1 = \frac{1}{T}\int_0^T X_1(t)\, dt,
\]
\item Long-term average renewable level $\bar{X}_4$,
\item Maximum import amplitude $\max_{t\in[0,T]} X_3(t)$.
\end{itemize}

\begin{figure}[H]
	\begin{center}
		\includegraphics[scale=0.8]{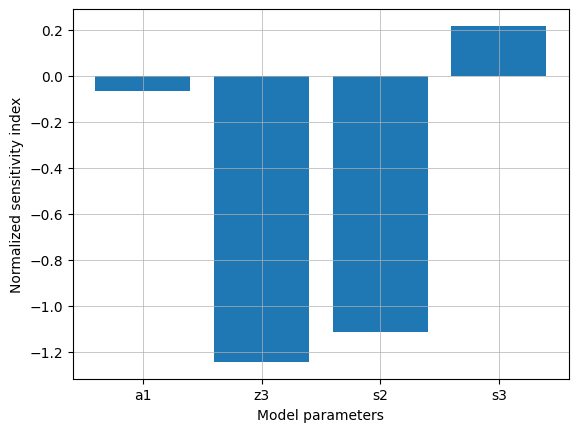}
		\caption{Sensitivity of long-term average demand.}
		\label{fig:placeholder-6}
	\end{center}
\end{figure}

Figure~6 illustrates the normalized sensitivity indices of selected parameters.
The results show that demand growth parameters ($a_1$) and supply responsiveness parameters $z_3$) represent the strongest factors which determine the system's long-term performance. The import-related parameters  $s_2$ and $s_3$ specifically influence the short-term increase of the import variable $X_3(t)$.

\paragraph{Noise sensitivity.}
To examine the impact of uncertainty, the noise intensities $\sigma_i$ are varied within the admissible ranges shown in Table~1. Increasing $\sigma_3$ produces wider fluctuation bands for imported energy, while larger $\sigma_4$ enhances renewable variability. Moderate noise intensities preserve boundedness and stability, consistent with the theoretical matrix inequality conditions.

Overall, the sensitivity analysis confirms that system stability depends on the balance between intrinsic growth parameters and stochastic perturbation strength. Deterministic equilibrium predictions remain qualitatively valid under moderate noise but may underestimate transient amplification and volatility effects.

\subsection{Model Long Term Implications and Possible Extensions }
The results of persistence in Section 5 and some numerical evidence show that the system is able to sustain strictly positive levels of demand and renewable resources in the long term average under a moderate intensity of noise. At the same time, the stability and persistence of the stochastic system depend on the balance between the intrinsic growth mechanisms and the strength of random perturbation. Therefore, when the rates of dispersive transport become too big compared to the stabilizing forces, the matrix inequalities may go unbounded, the conditions may break down, and the stability of the system cannot be guaranteed. Therefore, an emphasis should be placed on considering different levels of uncertainty into planning treatments, rather than merely following deterministic equilibrium predictions.

The Euler–Maruyama approach successfully reproduces the analytical results. The resulting trajectories are continuous, effectively showing out of the chaotic equilibrium, and exhibit stability with the inequalities developed. So, from this kind of numerical analysis, there comes proof to know that stochastic perturbations rewrite energy-system dynamics, bringing about realistic long-term persistence in the competitive spectrum, quite some stability in thought-to-be stabilized behavior with long-term projections against predictions of determinism.

The aforementioned framework enables any and all possible extensions that can be used in further enhancing the realism and practicability of the model. Delays deserve to be incorporated, representing investment lags and the development of infrastructure, while little jump disturbances, contingent jumps driven by Levy processes, may be allowed, aiming at abrupt and significant shocks: policy changes, disruptive supply chains, or specific hazards. Stochastic optimal control may find its application in different settings like those involving renewable subsidies, import regulation, or demand-side management policies. Further, parameter estimation and calibration using real market data and data on renewable production can be carried out in order to do the quantification of validity and, possibly via model improvements, to allow an enhancement of prognostic relevancy.\\

\section{Conclusion}

A stochastic energy supply and demand model was presented in this paper, extending to renewable integration and details. Given the case of a nonlinear deterministic framework, the model was extended to an Itô stochastic differential system, giving a fluency of inclusion of uncertainties regarding the market, environmental variability, emission reduction, technological change, and policy changes simultaneously. The stochastic formulation delivers a more plausible representation of an energy system that operates under uncertain economic and environmental conditions.

Besides, the solution has been thoroughly analyzed using mathematical methods to establish the global existence and uniqueness of positive components. This does not stop the model in any time from being well defined and physically meaningful. However, depending on the specified range of parameters, there are several possible levels for the rate of almost sure moment boundedness and Lyapunov asymptotic stability, leading to selective identification of the ones that are satisfactory. The utility and relevance of the results lie in the understanding that multiplicative stochastic perturbations may lead to destabilization of systems not necessarily. In certain cases, stochastic effects contribute to the stabilization and modify the effective structure of equilibrium, and the analysis of stochastic persistence also suggests that energy and renewable resources must stay positive in long-term average, hence fostering sustainability vis-a-vis moderate noise intensities.

As numerical simulations further support the theoretical findings, the numerical experiments, which included both convergence testing and sensitivity analysis, confirmed the accuracy of the numerical approximation while revealing that demand growth and supply responsiveness served as major factors that determined the system's long-term performance. In the trajectories, volatility is realistic; amplification is transient; and they remain bounded and positive. It is numerically shown that in the presence of stochastic perturbations, the transient dynamics are modified toward long-term averages compared to deterministic predictions, and the system stability criterion may hold under certain conditions.

In conclusion, the above results prove that stochastic variations are crucial in shaping the long term dynamics of energy systems. Thereby, incorporating uncertainty into supply–demand models is essential for achieving realistic system representation and effective energy planning.

Stochastic model validation, renewable funding and policy networks, and the control of these types of systems are the principal areas for future investigation. Enhancement of the model regarding delayed systems, jump systems, or multisite regional networked energy systems appear promising targets too.\\

\section*{Acknowledgment}
The authors express sincere thanks and appreciation to their respective institutions for the provision of good working environment and all forms of support. In addition, they thank the reviewers for their valuable feedback and constructive criticism which played a vital role in the completion of this research.

\section*{Funding} The work of S. Noeiaghdam was funded by the High-Level Talent Research Start-up Project Funding of Henan Academy of Sciences (Project No. 241819246).

\section*{Declaration of competing interest}
The authors declare no known competing financial or conflict of interests regarding this work.

\noindent \textbf{Author Contributions:} S.O.E., S.N., L.H., S.k.  and I.D.E wrote the main manuscript text, edited, software and coding and prepared the results. All authors reviewed the manuscript.

\section*{Data availability}
All the data used or generated during the current study is embedded in this published article.

\section*{Declaration of Generative AI}
The author conducted content analysis through Grammarly software during the work. The authors used the tool, which they reviewed and edited before submitting the final version of this article.

\section*{Ethics}
Not applicable.

\end{document}